\documentclass{amsart}
\usepackage{color}
\usepackage{filecontents}
\usepackage{mathrsfs} %\mathscr
\usepackage[all]{xy}
\usepackage[backref=page]{hyperref} %\cite[thm#]{} %\label{key} %\ref{key}
\hypersetup{
    colorlinks=true,
    linkcolor=blue,
    citecolor=red,
    }

\usepackage[capitalize,nameinlink]{cleveref}
\numberwithin{equation}{section}
\theoremstyle{plain}
\newtheorem{thm}[equation]{Theorem}
\newtheorem{prop}[equation]{Proposition}
\newtheorem{lem}[equation]{Lemma}
\newtheorem{cor}[equation]{Corollary}
\theoremstyle{definition}

\newtheorem{rem}[equation]{Remark}
\newtheorem*{rem*}{Remark}

\newcommand{\ST}{such that }

\DeclareMathOperator{\Char}{char} %characteristic
\DeclareMathOperator{\Herm}{Herm} %hermitian

\DeclareMathOperator{\Id}{Id}
\DeclareMathOperator{\Ind}{ind}
\DeclareMathOperator{\Int}{Int}

\DeclareMathOperator{\Nrd}{Nrd}

\DeclareMathOperator{\Per}{per}

\DeclareMathOperator{\Rank}{Rank}
\DeclareMathOperator{\Sym}{Sym}

\DeclareMathOperator{\op}{op}

\begin{document}
\title%[Hermitian u-invariants over semi-global fields]
{Hermitian $u$-invariants over function fields of $p$-adic curves}
\author{Zhengyao Wu}
\date{} %no date
\address{
Department of Mathematics and Computer Science\\
Emory University\\
400 Dowman Drive, W401\\
Atlanta, GA 30322}
\email{zwu22@emory.edu}
\subjclass[2010]{Primary: 11E39. Secondary: 14H05, 16W10}
\keywords{hermitian form, u-invariant, p-adic curve}
\begin{abstract}
Let $p$ be an odd prime. 
Let $F$ be the function field of a $p$-adic curve. 
Let $A$ be a central simple algebra of period 2 over $F$ with an involution $\sigma$. 
There are known upper bounds for the u-invariant of hermitian forms over $(A, \sigma)$. 
In this article we compute the exact values of the u-invariant of
hermitian forms over $(A, \sigma)$.
\end{abstract}
\maketitle
\tableofcontents
\section{Introduction}
%Let $K$ be a field. 
Let $A$ be a central simple algebra over a field $K$. 
Let $\sigma$ be an involution on $A$. 
Let $k=K^{\sigma}=\{x\in K~|~\sigma(x)=x\}$. 
Suppose $\Char k\ne 2$. 
%By Wedderburn's theorem, there exists a central division algebra $D$ over $K$ and an integer $m>0$ such that $A\simeq M_m(D)$. 
Suppose $\varepsilon\in\{1,-1\}$. 
%We fix above hypotheses throughout this paper. 
%If $h: A^n\times A^n\to A$ is an $\varepsilon$-Hermitian space over $(A, \sigma)$, we say that $h$ has rank $n$. 
If $V$ is a finitely generated right $A$-module and $h: V\times V\to A$ is an $\varepsilon$-hermitian space over $(A,\sigma)$, the \textit{rank} of $h$ is defined to be $\Rank(h)=\dfrac{\dim_K(V)}{\deg(A)\Ind(A)}$. 
Let $\Herm^{\varepsilon}(A,\sigma)$ denote the category of $\varepsilon$-hermitian spaces over $(A,\sigma)$. 
The Hermitian u-invariant \cite[2.1]{Mah} of $(A,\sigma, \varepsilon)$ is defined to be:  
\[
u(A, \sigma, \varepsilon)
=\sup\{n|\text{there exists an anisotropic }h\in\Herm^{\varepsilon}(A,\sigma), \Rank(h)=n.
\}\]

Suppose that $\sigma$ and $\tau$ are involutions on $A$. 
Mahmoudi has proved that \cite[2.2]{Mah} if $\sigma$ and $\tau$ are of the same type, then $u(A, \sigma, \varepsilon) 
= u(A, \tau, \varepsilon)$; 
if $\sigma$ is orthogonal and 
$\tau$ is symplectic, then $u(A, \sigma, \varepsilon) 
= u(A, \tau, -\varepsilon)$; if $\sigma$ is unitary, then $u(A, \sigma, 1)  = 
u(A, \sigma, -1)$.
Thus we have only three types of Hermitian $u$-invariants \cite[2.3]{Mah}, we denote: 
\[u(A, \sigma, \varepsilon)=\left\{
\begin{array}{ll}
u^+(A),&\text{if } \varepsilon=1\text{ and }\sigma \text{ is orthogonal,}\\
&\text{or, } \varepsilon=-1\text{ and }\sigma \text{ is symplectic};\\
u^-(A),&\text{if } \varepsilon=-1\text{ and }\sigma \text{ is orthogonal,}\\
& \text{or, }\varepsilon=1\text{ and }\sigma \text{ is symplectic};\\
u^{0}(A/k),&\text{ if }\sigma\text{ is a unitary }K/k\text{-involution.}\\
\end{array}\right.\]
where $u^+$ is called the \textit{orthogonal} Hermitian u-invariant, $u^-$ is called the \textit{symplectic} Hermitian u-invariant and $u^0$ is called the \textit{unitary} Hermitian u-invariant.

In \cref{sec3}, we provide upper bounds for Hermitian u-invariants of division algebras with Springer's property over $\mathscr A_i(2)$-fields. 
{For definitions of $\mathscr A_i(2)$-fields and  Springer's property, see the beginning of \cref{sec3}.} 

\begin{thm}\label{thm1} 
Let $D$ be a division algebra over a field $K$ with an involution  $\sigma$.  
Suppose $k=K^{\sigma}$, $\Char k\ne 2$, $\varepsilon\in\{1,-1\}$ and $d=\deg(D)$. 
Suppose $k$ is an $\mathscr A_i(2)$-field and $D$ satisfies the Springer's property. 

\noindent
(i) If $\sigma$ is of the first kind, then  $u^+(D)\le(1+\frac{1}{d})2^{i-1}$ and $u^-(D)\le(1-\frac{1}{d})2^{i-1}$; 

\noindent
(ii) If $\sigma$ is of the second kind, then $u^0(D/k)\le 2^{i-1}$. 
\end{thm}

Let $p$ be an odd prime number. 
Let $F$ be the function field of a smooth projective geometrically integral curve over a $p$-adic field. The field $F$ is also called a \textit{semi-global} field. 
Let $D$ be a central division $F$-algebra with an involution $\sigma$ of the first kind. 
Suppose $D\ne F$. 
As a consequence of an inequality of Mahmoudi \cite[3.6]{Mah} with $u(F)=8$ (\cite{PS10} or \cite{Heath-Brown} and \cite{L}), $u^+(D)\le 27$ and $u^-(D)\le 10$. 
Parihar and Suresh \cite{PS} have proved that $u^+(D)\le 14$ and $u^-(D)\le 8$. 
We obtain exact values of Hermitian u-invariants: 
%
%\begin{cor}
%Let $K$ be a $p$-adic field and $F$ the function field of a variety over $K$.
%Let $D$ be a quaternion division  over $F$.  If $n$ is the dimension of $X$, 
%then $u^+(D) \le  3 \cdot 2^{n}$ and $u^-(D) \le 2^n$. 
%%If $D$ has an involution of 
%%second kind, then $u^0(D) \le $.
%\end{cor}

\begin{thm}\label{u}
Let $F$ be the function field of a curve over a $p$-adic field with $p \ne 2$. 
Let $D$ be a central division algebra over $F$. Let $L/F$ be a quadratic extension. 

\noindent
(1) if $D$ is quaternion, then $u^+(D)=6$ and  $u^-(D)=2$; 

\noindent
(2) if $D$ is quaternion and $D\otimes_FL$ is division, then $u^0(D\otimes_F L/F) = 4$; 

\noindent
(3) if $D$ is biquaternion, then $u^+(D)=5$ and $u^-(D)=3$. 
\end{thm}

%Let $k$ be a field of characteristic not equal to 2. 
Let $A$ be a central simple algebra over a field $k$.
Suppose $\Char k\ne 2$ and $\Per(A)=2$.  
Then, by a special case \cite{Mer81} of the Merkurjev-Suslin theorem \cite{MS}, $A$ is Brauer equivalent to $H_1 \otimes \cdots \otimes H_n$
for some quaternion algebras  $H_1, \cdots H_n$ over $k$. 
Let $K/k$ be a quadratic extension. 
In \cite{PS}, upper bounds for $u^+(A)$, $u^-(A)$, $u^0(A \otimes_{ k}K/k)$ are given and they depend only on $u(k)$ and $n$. 
In \cref{app2}, we obtain sharper upper bounds for these Hermitian u-invariants. In fact we prove the following  

\begin{thm}\label{thm2}
Let $A$ be a central simple algebra over a field $k$.
Suppose $\Char k\ne 2$ and $\Per(A)=2$. 
Suppose $A$ is Brauer equivalent to $H_1 \otimes \cdots \otimes H_n$
for $n$ quaternion algebras $H_1, \cdots H_n$ over $k$. 

\noindent
(1) $u^+(A)\le (\frac{4}{5}+\frac{1}{5}(\frac{9}{4})^{n})u(k)$;

\noindent
(2) $u^-(A)\le (-\frac{1}{5}+\frac{1}{5}(\frac{9}{4})^{n})u(k)$;

\noindent
(3) $u^0(A\otimes_k K/k)\le (\frac{1}{5}+\frac{3}{10}(\frac{9}{4})^{n})u(k)$ 
for all quadratic extension $K/k$. 
\end{thm}

\noindent
\textbf{Acknowledgements}. 

The author acknowledges partial support from the NSF-DMS grant 1201542 and the NSF-FRG grant 1463882. 
The author thanks Professor V.~Suresh for his thorough detailed instructions and Professor R.~Parimala for helpful discussions. 
The author thanks Prof.~Wadsworth for his helpful comments. 
%He also thanks Emory University, where he studies mathematics. 

\section{Preliminaries}\label{pre}
Let $K$ be a field. 
Let $A$ be a central simple algebra over $K$ with an involution $\sigma$. 
Let $k=K^{\sigma}$. 
{We suppose $\Char(k)\ne 2$ throughout the paper.}
%By Wedderburn's theorem, $A \simeq M_m(D)$ for a central division algebra $D$ over $K$.
%Suppose $\tau$ is an involution on $A$, $D$ also has an involution $\tau$ of same kind as $\sigma$ \cite[3.1, 3.11, 3.20]{inv}. 
Let $V$ be a finitely generated right $A$-module and $\varepsilon\in \{ 1, -1 \}$. 
A map $h: V\times V\to A$ is called \textit{an $\varepsilon$-hermitian form} over $(A, \sigma)$ if $h$ is bi-additive; $h(xa, yb)=\sigma(a)h(x,y)b$ for all $a,b\in A$, $x,y\in V$; and $h(y,x)=\varepsilon\sigma(h(x,y))$ for all $x,y\in A$. 
We call $h$ \textit{an $\varepsilon$-hermitian space} if given $h(x,y)=0$ for all $x\in V$, we have $y=0$. 
We say that $h$ is \textit{isotropic} if there exists $x\in V$, $x\ne 0$ such that $h(x,x)=0$; otherwise we say that $h$ is \textit{anisotropic}. 
 
\begin{lem}[Morita Invariance]\label{Morita}
{Let $K$, $A$, $\sigma$, $k$ be as before.}
Suppose $A \simeq M_m(D)$ for a central division algebra $D$ over $K$.
Suppose $\sigma$ is an involution on $A$ and $\varepsilon\in\{1,-1\}$. 
Then there exists an involution $\tau$ on $D$ and $\varepsilon_0\in\{1,-1\}$ such that 
% of same kind as $\sigma$ \cite[3.1, 3.11, 3.20]{inv}. 
$u(A,\sigma,\varepsilon)=u(D,\tau,\varepsilon\varepsilon_0)$. 

Furthermore, $u^+(A)=u^+(D)$,  $u^-(A)=u^-(D)$ and $u^0(A/k)=u^0(D/k)$. 
\end{lem}
\begin{proof}
It is a consequence of \cite[ch.~I, 9.3.5]{Knus} and \cite[4.2]{inv}. 
\end{proof}

From now on, we mostly focus on central division algebras. 

\begin{lem}\label{local}
Let $D$ be a central division algebra over a field $K$ with an involution $\sigma$. 
Let $k=K^{\sigma}$, $\Char k\ne 2$. 
Suppose $k$ is a non-archimedean non-dyadic local field.

\noindent
(1) If $\sigma$ is of the first kind and $D\ne k$, then $u^+(D)=3$, $u^-(D)=1$. 

\noindent
(2) If $\sigma$ is of the second kind, then $u^0(D/k)=2$. 
\end{lem}

\begin{proof}
See \cite[Thm.~1, Thm.~3]{T} and \cite[ch.~10, 2.2]{Sch}.
\end{proof}

We fix the following {notations from 2.3 to 2.9}. 
Let $(k,v)$ be a complete discrete valued field
with residue field $\overline k$, $\Char \overline{k}\ne 2$. 
Let $D$ be a finite-dimensional division $k$-algebra with center $K$ with an involution $\sigma$ such that $K^{\sigma}=k$. 
By \cite[ch.~II, 10.1]{CF}, $v$ extends to a valuation $v'$ on $K$ such that $v'(x)=\frac{1}{[K:k]}{v}(N_{K/k}(x))$ for all $x\in K^*$.
By \cite{W1}, $v'$ extends to a valuation $w$ on $D$ such that $w(x)=\frac{1}{\Ind(D)}{v'}(\Nrd_{D/K}(x))$ for all $x\in D^*$.
Since $\Nrd_{D/K}(x)=\Nrd_{D/K}(\sigma(x))$, we have $w(\sigma(x))=w(x)$ for all $x\in D$. 
Let $R_w=\{x\in D~|~w(x)\ge 0\}$ and $\mathfrak m_w=\{x\in D~|~w(x)> 0\}$. 
Let $\overline D=R_w/\mathfrak m_w$ be the residue division algebra (see \cite[13.2]{Rei}) of $(D, w)$ over $\overline k$ with involution $\overline{\sigma}$ such that $\overline{\sigma}(\overline x)=\overline{\sigma(x)}$ for all $x\in R_w$, where $\overline{x}=x+\mathfrak m_w$. 
Let $h\colon V\times V\to D$ be an $\varepsilon$-hermitian space  over $(D,\sigma)$. 
{By \cite[Ch.~I, 6.2]{Knus}, $V$ is free with an orthogonal basis $\{e_1,\ldots, e_n\}$ such that $h(e_i, e_i)=a_i$ for some $a_i\in D$ with $\sigma(a_i)=\varepsilon a_i$ for all $1\le i\le n$;  and $h(e_i, e_j)=0$ for all $1\le i\le n$, $1\le j\le n$ and $i\ne j$. 
We denote $h=\langle a_1,\cdots,a_n\rangle$. }
If $w(a_i)=0$ for all $1\le i\le n$, then $\overline{h}=\langle \overline{a}_1,\cdots,\overline{a}_n\rangle\in\Herm^{\varepsilon}(\overline{D}, \overline{\sigma})$. 
%The following is a hermitian analogue of Springer's theorem for quadratic forms over a complete discrete valued field. 
%Up to isometry, we may assume that any $h\in\Herm^{\varepsilon}(D,\sigma)$ has diagonal entries with $w$-value either $0$ or $\frac{1}{e}$ \cite[2.20]{L2}. 
Let $t_D$ be a parameter of $(D, w)$. 
By \cite[2.7]{L2}, there exists $\pi_D\in D$ such that 
$w(\pi_D)\equiv w(t_D)\mod{2w(D^*)}$ and $\sigma(\pi_D)=\varepsilon'\pi_D$ for some  $\varepsilon'\in\{1,-1\}$. %such that $w(\pi_D)=\frac{1}{e}$. 
Larmour proved the following hermitian analogue of a theorem of Springer. 

\begin{prop}[{\cite[3.4]{L1} or \cite[3.27]{L2}}]\label{Lamour}
{Let $k$, $v$, $D$, $K$, $\sigma$, $w$, $h$, $\pi_D$ and $\varepsilon'$ be as above.} There exist $h_1\in\Herm^{\varepsilon}(D, \sigma)$, $h_2\in \Herm^{\varepsilon\varepsilon'}(D, \Int(\pi_D)\circ\sigma)$, with $h\simeq h_1\perp h_2\pi_D$ and each diagonal entries of $h_1$ and $h_2$ have $w$-value $0$. 
Further, the following are equivalent: 
(i) $h$ is isotropic; 
(ii) $h_1$ or $h_2$ is isotropic; 
(iii) $\overline h_1$ or $\overline h_2$ is isotropic. 
\end{prop}

\begin{cor}\label{Lamour2}
$u(D, \sigma, \varepsilon)=u(\overline D, \overline{\sigma}, \varepsilon)+u(\overline D, \overline{\Int(\pi_D)\circ\sigma}, \varepsilon\varepsilon')$. 
\end{cor}

\begin{proof}
Suppose $h\in \Herm^{\varepsilon}(D, \sigma)$ and $h\simeq h_1\perp h_2\pi_D$ as in \cref{Lamour}.  
Since $\Rank(h)=\Rank(h_1)+\Rank(h_2)=\Rank(\overline{h_1})+\Rank(\overline{h_2})$, if $\Rank(h) > 
u(\overline D, \overline{\sigma}, \varepsilon)  +  u(\overline D, \overline{\Int(\pi_D)\circ\sigma}, \varepsilon\varepsilon')$, 
then $\Rank(\overline{h_1})>u(\overline D, \overline{\sigma}, \varepsilon)$ or $\Rank(\overline{h_2})>u(\overline D, \overline{\Int(\pi_D)\circ\sigma}, \varepsilon\varepsilon')$. 
Then $\overline h_1$ or $\overline h_2$ is isotropic. 
By \cref{Lamour}, $h$ is isotropic. 
Hence $u(D, \sigma, \varepsilon)\le u(\overline D, \overline{\sigma}, \varepsilon)+u(\overline D, \overline{\Int(\pi_D)\circ\sigma}, \varepsilon\varepsilon')$. 

Conversely, 
suppose $g_1=\langle a_1,\cdots, a_m\rangle\in\Herm^{\varepsilon}(\overline D, \overline{\sigma})$ such that $\overline{\sigma}(a_i)=\varepsilon a_i$, $m=u(\overline D, \overline{\sigma}, \varepsilon)$ and $g_1$ is anisotropic. 
Since $a_i\ne 0$, there exists $b_i\in R_w$, $w(b_i)=0$ such that $\overline{b_i}=a_i$. 
Let $c_i=\frac{1}{2}(b_i+\varepsilon\sigma(b_i))$. 
Then $\sigma(c_i)=\varepsilon c_i$ and $\overline{c_i}=a_i$. 
Let $h_1=\langle c_1,\cdots, c_m\rangle\in\Herm^{\varepsilon}(D, \sigma)$. 
Then $\overline{h_1}=g_1$ and by \cite[2.3]{L1}, $h_1$ is anisotropic. 

Suppose $g_2=\langle a_{m+1},\cdots, a_{m+n}\rangle\in\Herm^{\varepsilon\varepsilon'}(\overline D, \overline{\Int(\pi_D)\circ\sigma})$ such that  
$g_2$ is anisotropic. 
Similar to the previous paragraph, 
there exists $h_2\in\Herm^{\varepsilon\varepsilon'}(D, \Int(\pi_D)\circ\sigma)$ such that 
$\overline{h_2}=g_2$ and $h_2$ is anisotropic. 

By \cref{Lamour}, $h=h_1\perp h_2\pi_D$ is anisotropic and $\Rank(h)=m+n$. 
%by the same method as the proof of \cref{dvr}, 
$u(D, \sigma, \varepsilon)\ge u(\overline D, \overline{\sigma}, \varepsilon)+u(\overline D, \overline{\Int(\pi_D)\circ\sigma}, \varepsilon\varepsilon')$. 
\end{proof}

\begin{lem}\label{choice_of_pid}
Suppose $D$ is ramified at the discrete valuation of $k$,  {$Z(D)=k$ and $\Per(D)|2$. }
Then there exist an involution $\sigma$ on $D$ of first kind and elements $\alpha, {\pi_D} \in D$ satisfying the following conditions: 

(a) $\overline{\sigma}$ is an involution of the second kind;

(b) $\alpha^2$ is  a unit at the valuation $v$ of $k$ and $Z(\overline{D}) = \overline{k}( \overline{\alpha})$;

(c) $\pi_D$ is a parameter of $D$, $\sigma(\pi_D) =  \pm \pi_D$ and 
$\overline{\Int(\pi_D)\circ\sigma}$ is of the first kind. 
\end{lem}

\begin{proof} Suppose  $D$ is ramified.  
Then $D$ is Brauer equivalent to $D_0\otimes(u,\pi)$ with 
$D_0$ a central division algebra over $k$ unramified at $v$, $\pi \in k^*$ a parameter
 and $u\in k^*\setminus{k^*}^2$ a unit at $v$. 
Furthermore, $\overline{D}$ Brauer equivalent to $ \overline{D_0} \otimes \overline{k}(\overline{\sqrt{u}})$ and ${Z(\overline{D})=\overline{k}(\overline{\sqrt{u}})}$ 
 \cite[8.77]{TW15}. 

(a) By \cite[prop.~4]{CDTWY}, the nontrivial automorphism of $Z(\overline{D})/\overline{k}$ extends to an involution on $\overline{D}$ of the second kind and it can be lifted to an involution $\sigma$ on $D$ of the first kind. 

(b) Since $k$ is complete, by \cite[p.~53, Lem.~1]{CDTWY}, there exists $\alpha \in D$ such that $\alpha^2\in Z(D)$, $\overline{\alpha} \in Z(\overline{D})$ corresponds $\overline{\sqrt{u}}$ in the isomorphism ${Z(\overline{D})=\overline{k}(\overline{\sqrt{u}})}$ and $\sigma(\alpha)=-\alpha$. 

(c) By \cite[prop.~1.7]{JW}, there exists a parameter $t_D \in D$ such that $\overline{\Int(t_D)}$ is the non-trivial $Z(\overline{D})/\overline{k}$-automorphism.
Since $\overline{\sigma}$ is of the second kind and $\overline{\Int(t_D)}$ induces the non-trivial automorphims of $Z(\overline{D})$, $\overline{\Int(t_D)\circ\sigma}$ is of the first kind.
  Since $\sigma$ is an involution, $w(t_D)  = w(\sigma(t_D))$ and hence 
  $\overline{\sigma(t_D)t_D^{-1}}\ne 0\in\overline{D}$.
  
Case 1: Suppose that $\overline{\sigma(t_D) t_D^{-1}} = 1 $. Let $\pi_D = t_D + \sigma(t_D)$. 
Then $\sigma(\pi_D) = \pi_D$.   
  Since $\pi_Dt_D^{-1} =  1 +  \sigma(t_D) t_D^{-1}$, $\overline{\pi_Dt_D^{-1}} = 1 + \overline{\sigma(t_D)t_D^{-1}}
  = 1 + 1 = 2\ne 0$. Hence $w(\pi_D) = w(t_D)$.
  Since $\overline{\pi_Dt_D^{-1}} = 2$,   
  $\overline{\Int(\pi_D)\circ\sigma} = \overline{\Int(t_D)\circ\sigma}$ and hence  
  $\overline{\Int(\pi_D)\circ\sigma}$ is of the first kind. 
Thus $\pi_D$ satisfies (c). 
  
Case 2: Suppose that $\overline{\sigma(t_D) t_D^{-1}} \ne 1$. Let $\pi_D =  \alpha t_D - \sigma(\alpha t_D )$.
  Then $\sigma(\pi_D) = - \pi_D$.
  We have $\pi_Dt_D^{-1} = \alpha -  \sigma(t_D)\sigma(\alpha)t_D^{-1}$.
  Since $\overline{\sigma(\alpha)} = - \overline{\alpha}$ and $ \overline{t_D\alpha t_D^{-1}} 
  = -\overline{\alpha}$, we have  
$$
\begin{array}{lll}
\overline{\pi_Dt_D^{-1}}
&=&\overline{\alpha} - \overline{\sigma(t_D)\sigma(\alpha) t_D^{-1}}\\
&=&\overline{\alpha } + \overline{\sigma (t_D)t_D^{-1}t_D\alpha t_D^{-1}}\\
 &=&\overline{\alpha} + \overline{\sigma(t_D)t_D^{-1}} (-\overline{\alpha})\\
&=  &( 1-  \overline{\sigma(t_D)t_D^{-1}})\overline{\alpha} \ne 0
\end{array}
$$
Hence $w(\pi_D) = w(t_D)$.  Since $\overline{\sigma(\alpha)} = - \overline{\alpha}$,
$\alpha^2 \in k$ 
and $\overline{t_D\alpha t_D^{-1}} = -\overline{\alpha}$, 
we have $\overline{\sigma(t_D)\alpha \sigma(t_D)^{-1}} = -\overline{\alpha}$  and

$$
  \begin{array}{lll}
&&\overline{(\pi_D\alpha \pi_D^{-1} + \alpha) \pi_Dt_D^{-1}} \\
& = & \overline{\pi_D \alpha t_D^{-1}} + 
\overline{\alpha \pi_D t_D^{-1}} \\
& = &  \overline{(\alpha t_D - \sigma(t_D)\sigma(\alpha)) \alpha t_D^{-1}}
+ \overline{\alpha (\alpha t_D - \sigma(t_D) \sigma(\alpha))t_D^{-1}}\\
& = &  \overline{\alpha t_D\alpha t_D^{-1}}- \overline{\sigma(t_D)\sigma(\alpha)\alpha t_D^{-1}}
+ \overline{\alpha^2}
+ \overline{\alpha \sigma(t_D) \alpha t_D^{-1}} \\
& = & -\overline{\alpha^2} + \overline{\sigma(t_D) \alpha^2 t_D^{-1}} + \overline{\alpha^2}
+ \overline{\alpha (\sigma(t_D) \alpha \sigma(t_D)^{-1})\sigma(t_D)t_D^{-1}} \\
& = & \overline{\alpha^2 \sigma(t_D)t_D^{-1}} -   \overline{\alpha^2 \sigma(t_D)t_D^{-1}} = 0.
\end{array}
$$
Since $\overline{\pi_Dt_D^{-1}} \ne 0$, $\overline{\pi_D\alpha\pi_D^{-1} + \alpha} = 0$
and hence $\overline{(\Int(\pi_D)\circ\sigma)}(\overline{\alpha}) =\overline{\alpha}$.
Thus $\pi_D$ satisfies (c). 

To summarize, $\sigma$, $\alpha$ and $\pi_D$ satisfy required properties. 
\end{proof}

\begin{cor}\label{Lamour3}
Suppose {$Z(D)=k$ and $\Per(D)=2$}.  

\noindent
(1) If $D$ is unramified at the discrete valuation of $k$, then 
$$u^{+}(D)=2u^{+}(\overline D)  \text{ and }   u^{-}(D)=2u^{-}(\overline D).$$ 

\noindent
(2) If $D$ is ramified at the discrete valuation of $k$,  
 then 
$$u^{+}(D)=u^0(\overline{D}/\overline{k})+u^{+}(\overline{D}) \text{ and } u^{-}(D)=u^0(\overline{D}/\overline{k})+u^{-}(\overline{D}).$$
\end{cor}

\begin{proof}
 Suppose  $D$ is unramified. Then we can take $\pi_D = \pi$, where $\pi$ is a parameter of $k$. 
Since $\sigma(\pi) = \pi$, we have 
$\varepsilon' = 1$ and $\Int(\pi_D)\circ\sigma = \sigma$. 
Hence, by \cref{Lamour2}, we have the required result.
%$u(D,\sigma,\varepsilon)
%=u(\overline D, \overline{\sigma}, \varepsilon)+u(\overline D, \overline{\Int(\pi_D)\circ\sigma}, \varepsilon\varepsilon')
%=2u(\overline D, \overline{\sigma}, \varepsilon)
%$. 
%Then $u^{+}(D)=2u^{+}(\overline D)$ and $u^{-}(D)=2u^{-}(\overline D)$.  

Suppose $D$ is ramified. Then choose $\sigma$ and $\pi_D$ as in \cref{choice_of_pid}.
Then, by \cref{Lamour2}, we have the required result.
\end{proof}

%We fix the following notation for \cref{Lamour4}. 
Let $K/k$ be a quadratic extension.% and $\overline{K}$ the residue field of $K$.
Let $D$ be a central division algebra over $k$ with an involution $\sigma$ of the first kind.
Then $\sigma\otimes \iota$ is an involution on $D\otimes_k K$ of the second kind with 
$\iota$ being the non-trivial automorphism of $K/k$.  

{\begin{rem}\label{rem2.7}
Suppose $D\otimes_k K$ is division. 
Then there are three possibilities of ramification: 

(1) $K/k$ is unramified and $D\otimes_k K/K$ is unramified; 

(2) $K/k$ is unramified and $D\otimes_k K/K$ is ramified; 

(3) $K/k$ is ramified and $D\otimes_k K/K$ is unramified; 

We show that ``$K/k$ is ramified and $D\otimes_k K/K$ is ramified'' cannot happen. 

In fact, if $K/k$ is ramified, then $K = k(\sqrt{\pi})$ for some parameter $\pi \in k$
and $\overline{K} = \overline{k}$.  
If $D/k$ is unramified, then $D\otimes_k K/K$ is unramified.
Suppose $D/k$ is ramified. Then $D$ is Brauer equivalent to $D_0 \otimes (u, \pi)$ for some 
$D_0$ unramified on $k$ and $u \in k$ a unit at the valuation of $k$ \cite[8.77]{TW15}. 
Thus $D \otimes_k K$ is Brauer equivalent to $D_0 \otimes_k K$ and hence 
$D \otimes_k K/K$ is unramified.  

Consequently, (2) and (3) can be shortened as 

\textit{(2) $D\otimes_k K/K$ is ramified. (3) $K/k$ is ramified. }
\end{rem}}

\begin{rem}\label{rem2.8}
Suppose we are in case (2) of \cref{rem2.7}. 
Suppose $K=k(\sqrt{\lambda})$ and $D$ is Brauer equivalent to $D_0 \otimes (u, \pi)$ for some 
$D_0$ unramified on $k$ and $u \in k$ a unit at the valuation of $k$. 
Then $\overline{K} = \overline{k}(\overline{\sqrt{\lambda}})$, 
$Z(\overline{D}) = \overline{k}(\overline{\sqrt{u}})$ and $Z(\overline{D \otimes_k K}) = \overline{k}(\overline{\sqrt{u}}, \overline{\sqrt{\lambda}})$.
{Here $u$ and $\lambda$ are in different square classes of $k$, otherwise $(u, \pi)_K$ is split and hence $D\otimes_k K$ is unramified over $K$. }
Since $\overline{D \otimes_k K} = \overline{D} \otimes_{\overline{k}} \overline{K} 
= \overline{D} \otimes \overline{k}(\overline{\sqrt{u}}, \overline{\sqrt{\lambda}})$ and $\overline{D}$
has an involution of the first kind, $\overline{D \otimes_k K}$ has three possible types of involutions of 
second kind with fixed fields $\overline{k_1}=\overline{k}(\overline{\sqrt{u}})$, $\overline{k_2}=\overline{k}(\overline{\sqrt{\lambda}})$ and
$\overline{k_3}=\overline{k}(\overline{\sqrt{u\lambda}})$ respectively.
The corresponding $u^0$
are written as $u^0(\overline{D \otimes_k K}/\overline{k}_1)$, $u^0(\overline{D \otimes_k K}/
\overline{k}_2)$ and $u^0(\overline{D \otimes_k K}/\overline{k}_3)$.

\end{rem}

\begin{cor}\label{Lamour4} 
Let $K/k$ be a quadratic extension and let $\iota$ be the non-trivial automorphism of $K/k$. 
Let $D$ be a central division algebra over $k$ with an involution $\sigma$ of first kind such that $D \otimes_k K$ is division. 

    	(1) If $D\otimes_k K$ is  unramified  at the discrete valuation of $K$ and  $K/k$ is unramified, 
	then $u^0(D \otimes_k K/k)=2u^0(\overline{D} \otimes_{\overline{k}} \overline{K}/\overline{k})$. 

(2) If $D \otimes_k K/K$ is ramified, 
then $u^0(D \otimes_k K/k)=  u^0(\overline{D} \otimes_{\overline{k}} \overline{K}/\overline{k}_1)  
+ u^0(\overline{D} \otimes_{\overline{k}} \overline{K}/\overline{k}_3)$. 

	(3) If $K/k$ is ramified, then $u^0(D \otimes_k K/k)=u^+(\overline{D})+u^-(\overline{D})$. 
\end{cor}

\begin{proof} 
(1) Suppose $D$ is unramified and $K/k$ is unramified.   Then $\overline{D \otimes_k K} =
\overline{D} \otimes_{\overline{k}} \overline{K}$ and $\overline{K}/\overline{k}$ is a quadratic extension. 
%Let $\sigma_0$ be an involution on  $D$   of first  kind
%and   $\sigma = \sigma_0 \otimes \iota$, where $\iota$ is the non-trivial automorphism of 
%$K/k$.
Let  $\pi$ be a parameter of $k$.
Take $\pi_D=\pi$. Then  $\sigma(\pi_D) = \pi_D$ and  $\overline{\Int(\pi_D)\circ(\sigma\otimes\iota)}=\overline{\sigma\otimes\iota}$. 
By \cref{Lamour2}, $u^0(D \otimes_k K/k)=2u^0(\overline{D} \otimes_{\overline{k}} \overline{K}/\overline{k})$.  

(2) Suppose $D$ is ramified. 
Suppose $\sigma$ and $\pi_D$ are as in \cref{choice_of_pid}. 
{Then the fixed field of $\overline{\sigma\otimes \iota}$ is $\overline{k}_3$ and 
the fixed field of $\overline{\Int(\pi_D)\circ(\sigma\otimes \iota)}$ is $\overline{k}_1$ (where $\overline{k}_1$ and $\overline{k}_3$
are as in \cref{rem2.8}).}
Thus, by \cref{Lamour2}, we have 
$u^0(D \otimes_k K/k)=  u^0(\overline{D} \otimes_{\overline{k}} \overline{K}/\overline{k}_1)  
+ u^0(\overline{D} \otimes_{\overline{k}} \overline{K}/\overline{k}_3)$.

(3) Suppose $K/k$ is ramified. 
%Then $K = k(\sqrt{\pi})$ for some parameter $\pi \in k$
%and   $\overline{K} = \overline{k}$.  
%We have $D = D_0 \otimes (u, \pi)$ for some 
%$D_0$ unramified on $k$ and $u \in k$ a unit at the valuation of $k$ \cite[8.77]{TW15}. 
%Thus $D \otimes_k K  = D_0 \otimes_k K$. 
%Since $D \otimes_k K$ is division, $D \otimes_k K 
%\simeq D_0 \otimes_k K$ and $\deg(D)=\deg(D_0)$. 
Let $\sigma_0$ be an involution of the first kind on $D$ and
$\sigma\simeq \sigma_0\otimes \iota$, where $\iota$ is the non-trivial automorphism of $K/k$.
We have
$\overline{D\otimes_k K} = \overline{D}$ and $\overline{\sigma_0}\otimes \overline{\iota} = \overline{\sigma_0}$.
Let $\pi_D = \sqrt{\pi} \in K \subset D \otimes_k K$.
Then $\overline{\Int(\pi_D)\circ(\sigma_0\otimes \iota)} = \overline{\sigma_0}$. 
Thus, by  
 \cref{Lamour2}, $u(D\otimes_k K, \sigma,  \varepsilon) = u( \overline{D},  \overline{\sigma_0}, \varepsilon)
+ u(\overline{D}, \overline{\sigma_0}, -\varepsilon)$. 
Hence $u^0(D\otimes_k K/k)=u^+(\overline{D})+u^-(\overline{D})$. 
\end{proof}

We end this section with the following well known 
\begin{lem}
 \label{dvr}
 Let $k$ be a discrete valued field with residue field $\overline{k}$
 and completion $\widehat{k}$.  Suppose $\Char(\overline{k})\ne 2$. 
 Let $D$ be a division algebra over $k$ with center $K$. 
Let $\sigma$ be an involution on $D$ such that $K^{\sigma}=k$. 
 If $D \otimes \widehat{k}$ is division, 
 then $u(D, \sigma, \varepsilon) \ge u(D\otimes \widehat{k}, \widehat{\sigma}, \varepsilon)$, {where $\widehat{\sigma}=\sigma\otimes\Id_{\widehat{k}}$.}
\end{lem}

\begin{proof}   Let $v$ be the normalized discrete valuation on $k$, $\pi \in k$ and $ v(\pi)=1 $.
Since $D \otimes \widehat{k}$ is division,  $v$ extends to a valuation $w$ on $D$ and $ w|_{k}=v $. 
  Let $\varepsilon  = \pm 1$ and 
  $\Sym^\varepsilon(D, \sigma) = \{ x \in D  \mid \sigma(x) = \varepsilon  x\}$.
 Let $e_1, \cdots , e_r$ be a $k$-basis of $\Sym^{\varepsilon}(D, \sigma) $ such that $ w(e_{i})\ge 0 $ for all $ 1\le i\le r $.  
 Let $a \in \Sym^\varepsilon(D, \sigma) \otimes \widehat{k}$ 
 and write $ a= a_1e_1 + \cdots  + a_re_r$ with $a_i \in \widehat{k}$.
 Let $b_i \in k$ be such that $a_i  \equiv b_i$ modulo $\pi^{w(b) + 1}$
 and $b = b_1e_1 + \cdots + b_re_r \in \Sym^\varepsilon(D, \sigma)$. 
%, where $e$ is the ramification index $[w(D^*): v(k^*)]$. 
We have $ w(a-b)>w(b) $. 
Since $ ab^{-1}=(a-b)b^{-1}+1 $, we have $\overline{ab^{-1}} = 1 \in \overline{D \otimes \widehat{k}}$. 

%{If $w(a) = w(b)=0$, then $\overline{a}=\overline{b}$ and $\langle \overline{a}, -\overline{b}\rangle$ is isotropic, by \cref{Lamour}, $\langle a, -b\rangle$ is isotropic and hence $\langle a\rangle  \simeq \langle b\rangle\otimes\widehat{k} $. 
%
%When $w(a) = w(b)=1$, let $\pi_D$ be a parameter of $(D, w)$ such that $\sigma(\pi_D)=\varepsilon'\pi_D$ for some $\varepsilon'\in\{1,-1\}$. 
%Suppose $\langle a\rangle\simeq\langle a'\rangle\pi_D$ for some $a'\in \Sym^{\varepsilon\varepsilon'}(D, \sigma)\otimes \widehat{k}$ such that $w(a')=0$; 
%$\langle b\rangle\simeq\langle b'\rangle\pi_D$ for some $b'\in \Sym^{\varepsilon\varepsilon'}(D, \sigma)$ such that $w(b')=0$. 
%By the previous case, $\langle a'\rangle  \simeq \langle b'\rangle\otimes\widehat{k} $. 
%Hence $\langle a\rangle  \simeq \langle b\rangle\otimes\widehat{k}$.} 

{Let $s=ab^{-1}\in D \otimes \widehat{k}$. Then $w(s)=0$ and $\overline{s}=1$
$$\begin{array}{lllll}
a=sb&\implies&\widehat{\sigma}(a)=\sigma(b)\widehat{\sigma}(s)
&\implies&\varepsilon a=\varepsilon b\widehat{\sigma}(s)\\
&\implies& a=b\widehat{\sigma}(s)
&\implies& sb=b\widehat{\sigma}(s)\\
&\implies& s=(\Int(b)\circ\widehat{\sigma})(s)
&\implies&(\Int(b)\circ\widehat{\sigma})|_{\widehat{k}(s)}=\Id_{\widehat{k}(s)}.\\
\end{array}
$$
Since $\widehat{k}(s)$ is complete, by hensel's lemma, there exists $c\in \widehat{k}(s)$ such that $c^2=s$ and $\overline{c}^2=\overline{s}=1\in \overline{\widehat{k}(s)}$. Then
\[
\begin{array}{lll}
(\Int(b)\circ\widehat{\sigma})(c)=c&\implies&b\widehat{\sigma}(c)=cb\\
&\implies&a=sb=ccb=cb\widehat{\sigma}(c)\\
&\implies&\langle a\rangle\simeq \langle b\rangle\otimes\widehat{k}\\
\end{array}
\]}

Let $h$ be an $\varepsilon$-hermitian forms over $(D\otimes \widehat{k}, \sigma)$.
Since $D \otimes \widehat{k}$ is division,  $h = \langle \alpha_1, \cdots , \alpha_n\rangle $ for some 
$\alpha_i \in \Sym^\varepsilon(D, \sigma) \otimes \widehat{k}$.  For each $\alpha_i$, 
let $\beta_i \in \Sym^{\varepsilon}( D,\sigma)$ be such that $\langle \alpha_i\rangle  \simeq \langle \beta_i\rangle \otimes\widehat{k}$
and $h_0 = \langle \beta_1, \cdots , \beta_n\rangle $.  Then $h_0$ is an $\varepsilon$-hermitian form over
$(D, \sigma)$ and $h_0 \otimes \widehat{k} \simeq h$. If $h$ is anisotropic over 
$\widehat{k}$, then $h_0$ is anisotropic. 
In particular,  
$u(D, \sigma, \varepsilon) \ge u(D\otimes \widehat{k}, \sigma\otimes\Id, \varepsilon)$.
\end{proof}

\section{Division algebras over $\mathscr A_i(2)$-fields}\label{sec3}

A field $k$ is called an \textit{$\mathscr A_i(m)$-field} \cite[2.1]{L} if every system of $r$ homogeneous forms of degree $m$ in more than $rm^i$ variables over $k$ has a nontrivial simutaneous zero over a field extension $L/k$ such that $\gcd(m, [L:k])=1$.

Let $ A $ be a central simple algebra over a field $ K $ of characteristic not $ 2 $. 
Let $ \sigma $ be an involution on $ A $. Let $ k $ be the subfield of $ K $ consisting of elements of $ K $ fixed by $ \sigma $. 
We say that $ (A, \sigma) $ satisfies \textit{Springer's property} if for all odd degree field extensions $ L/k $ and for all anisotropic hermitian space $ h $ over $ (A, \sigma) $, the scalar extension $ h_{L} $ remains anisotropic over $ (A_{L}, \sigma_{L}) $. 

In \cite{PSS}, Parimala, Sridharan and Suresh have shown that Springer's property holds for  hermitian or skew-hermitian spaces over quaternion algebras with involution of the first kind. In \cite{Wu}, the author has shown that Springer's property holds for hermitian or skew-hermitian spaces over central simple algebras with involution of the first kind over function fields of $p$-adic curves.

Now we prove \cref{thm1}(i).
%\begin{thm}\label{a2fields}
%	Let $k$ be an $\mathscr A_i(2)$-field. 
%	Let $D$ be a central division algebra over $k$ with an involution of the first kind. 
%	If $D$ satisfies the Springer's property, then 
%	$$u^+(D)\le(1+\dfrac{1}{d})2^{i-1}\text{ and }u^-(D)\le(1-\dfrac{1}{d})2^{i-1},$$
%	where $d=\deg(D)$. 
%\end{thm}

\begin{proof} %Suppose $D$ has an involution of the first kind.
Let $\sigma$ be an orthogonal involution on $D$.
Let $\Sym^{\varepsilon} (D, \sigma) = \{x\in D~|~\sigma(x)=\varepsilon x\}$ and 
$r = \dim_k(\Sym^{\varepsilon}(D, \sigma))$. 
Then $r=d(d+\varepsilon)/2$ \cite[2.6]{inv}.
Let $e_1,\cdots, e_r$ be a $k$-basis of $\Sym^{\varepsilon}(D,\sigma)$. 
Let  $h$ be  an $\varepsilon$-hermitian form over $(D, \sigma)$ of rank $n>(1+\frac{\varepsilon}{d})2^{i-1}$.
Then for $x \in D^n$, we have  
 $$h(x,x)=q_1(x,x)e_1+\cdots+q_r(x,x)e_r,$$ 
with each $q_i$ a quadratic form  over $k$ in $d^2n$ variables \cite[proof of prop.~3.6]{Mah}. 
  
Since $k$ is an $\mathscr A_i(2)$-field and $d^2n>d(d+\varepsilon)2^{i-1}=r2^i$, there exists an odd degree extension $L/k$ such that $\{q_1,\cdots, q_r\}$ have a simultaneous nontrivial zero over $L$. 
Then $h_L$ is isotropic over $D_L$. 
By Springer's property, $h$ is isotropic over $D$. 
Hence $u(D,\sigma, \varepsilon)\le(1+\frac{\varepsilon}{d})2^{i-1}$. 

Similarly, if $\sigma$ is a symplectic involution on $D$, then $r=d(d-\varepsilon)/2$ and hence $u(D,\sigma, \varepsilon)\le(1-\frac{\varepsilon}{d})2^{i-1}$.
\end{proof}

{Next, we prove \cref{thm1}(ii).}
%\begin{thm}  
%\label{a2fields2}
%Let $k$ be an $\mathscr A_i(2)$-field. 
%Let $K/k$ be a quadratic extension.
%Let $D$ be a  central division algebra over $K$ with 
%an involution $\sigma$ of the second kind with $\sigma |_k = \Id$. 
%Suppose that  $D$ satisfies the Springer's property. 
%Then $u^0(D) \le 2^{i-1}$.
%\end{thm}

\begin{proof}
Let $\sigma$ be an involution on $D$ of the second  kind.
Let $\Sym(D) = \{ x \in D \mid \sigma(x) = x \}$. 
Then $\Sym(D)$ is vector space over $k$
and $\dim_k\Sym(D) = d^2$, where $d^2 = \dim_K(D)$.
Let $e_1,\cdots, e_{d^2}$ be a $k$-basis of $\Sym(D)$. 
Let  $h$ be a hermitian form over $(D, \sigma)$ of rank $n> 2^{i-1}$. 
Then, for $x \in D^n$, $h(x, x) \in \Sym(D)$ and
we have      
$$h(x,x)=q_1(x,x)e_1+\cdots+q_{d^2}(x,x)e_{d^2},$$ 
with  each $q_i$  a quadratic form  over $k$ in   $2d^2n$ variables. 
  
Since $k$ is an $\mathscr A_i(2)$-field and $2d^2n> 2 d^22^{i-1} = d^22^i$,  
there exists an odd degree extension $L/k$ such that $\{q_1,\cdots, q_{d^2}\}$ have a simultaneous nontrivial zero over $L$. 
In particular, $h_L$ is isotropic over $D_L$. 
By Springer's property, $h$ is isotropic over $D$. 
Hence $u^0(D/k) \le  2^{i-1}$. 
\end{proof}

\begin{cor}\label{quaternion}
If $D$ is a quaternion division algebra over an $\mathscr A_i(2)$-field $k$ with $ \Char k\ne 2 $ 
  and $\sigma$ is of the first kind, then $u^+(D)\le 3\cdot2^{i-2}$ and $u^-(D)\le 2^{i-2}$;  
\end{cor}

\begin{proof}
 Since, by  \cite[3.5]{PSS}, $(D, \sigma,\varepsilon)$ satisfies Springer's property, the corollary follows
 from \cref{thm1}(i).
  \end{proof}
  
\begin{cor}\label{global}
If $D$ is a quaternion division algebra  over   a global function field $k$, then $u^+(D)=3$, $u^-(D)=1$,
and $u^0(D/k) = 2$. 
\end{cor}

\begin{proof}
	By Chevalley-Warning theorem \cite{Chevalley, Warning}, every finite field is a $C_1$-field. 
	By Tsen-Lang theorem \cite{Lang}, every global function field is a $C_2$-field.
	Since every $C_2$-field is an $\mathscr A_2(2)$-field \cite[between 2.1 and 2.2]{L} and Springer's property is true over global fields by \cite[Lem.~5.7]{Wu}, by \cref{thm1}(i), 
	 $u^+(D) \le 3$ and $u^-(D) \le 1$ and  by \cref{thm1}(ii), $u^0(D) \le 2$.
	 The equality follows from \cref{dvr} and \cref{local}.
	\end{proof}

\begin{cor}\label{p-adic_varieties}
Let $F$ the function field of an integral variety $X$ over a $p$-adic field with $p\ne 2$.
Let $D$ be a quaternion algebra over $F$. 
If $\dim(X)=n$,
then $u^+(D) \le 3\cdot 2^n$ and $u^-(D) \le 2^{n}$.
\end{cor}

\begin{proof}
Since $D$ is a quaternion algebra, by \cite[3.5]{PSS}, $D$ satisfies 
the Springer's property. 
Since $\dim(X) = n$, 
by \cite{Heath-Brown} and \cite{L}, $F$ is a $\mathscr A_{n+2}(2)$-field.
Hence the corollary follows from \cref{quaternion}.
\end{proof}

\begin{cor}\label{semiglobal}
Let $F$ be a the function field of a $p$-adic curve with $ p\ne 2 $. 
Let $D$ be a division algebra over $F$ with an involution of the first kind. 

\noindent
(1) If $D$ is a quaternion division algebra, then $u^+(D)\le 6$ and $u^-(D)\le 2$. 

\noindent
(2) If $D$ is a biquaternion division algebra, then $u^+(D)\le 5$ and $u^-(D)\le 3$. 
\end{cor}

\begin{proof} By \cite[3.4]{Sal1, Sal1.5}, $\deg(D)=d=2$ or $4$. 
If $d = 2$, then $D$ is a quaternion algebra and result follows from 
\cref{p-adic_varieties}.
Suppose $ d = 4$.
By \cite[1.5]{Wu}, $D $ satisfies Springer's property.
Since $F$ is a $\mathscr A_{3}(2)$-field, the result follows from  \cref{thm1}(i).
\end{proof}

\begin{cor}\label{semiglobal2}
Let $F$ the function field of a $p$-adic curve with $ p\ne 2 $. 
Let $L/F$ be a quadratic extension.
Let  $D$  a quaternion division algebra over $F$ with an involution of the first kind.
Then $u^0(D\otimes_F L/F) \le 4$. 
\end{cor}

\begin{proof}  
Suppose $ L=F(\sqrt{\lambda}) $ is a quadratic field extension of $ F $. Let $ \Omega $ be the set of divisorial discrete valuations on $ F $. 
Let $ h $ be a hermitian form of rank $ n>4 $ over  $ D \otimes_F L $ with an $ L/F $-unitary involution $ \sigma $. 
We first show that $ h_{F_{v}} $ is isotropic for all $ v\in \Omega $, where $ F_{v} $ is the completion of $ F $ at $ v $.  

Case 1: Suppose $\lambda$ is a square in $F_v$. 
By \cite[ch.10, 6.3]{Sch}, $ L\otimes_FF_{v}\simeq F_{v}\times F_{v} $, the $ (L\otimes_FF_{v})/F_{v} $ unitary involution $ (\sigma|_L)\otimes_{F}\Id_{F_{v}} $ exchanges the two factors, $D\otimes_F F_v\simeq \Delta\times \Delta^{\op}$ for some central simple algebra $ \Delta $ over $ F_{v} $, and $h_{F_v}$ is hyperbolic. Then $ h_{F_{v}} $ is isotropic. 

Case 2: Suppose $\lambda$ is not a square in $F_v$. 
Then $ L\otimes_{F} F_{v} $ is a field with the nontrivial $  (L\otimes_{F} F_{v})/F_{v} $-automorphism $ (\sigma|_L)\otimes_{F}\Id_{F_{v}} $. Since $ D $ is a quaternion algebra over $ F $, $ D\otimes_FF_{v}$ is a quaternion algebra over $ L\otimes_{F} F_{v} $. So $ D\otimes_FF_{v}$ is either division or split. 

Subcase 2a: Suppose $\lambda$ is not a square in $F_v$ and $ D\otimes_FF_{v} $ is a central division algebra over the field $ L\otimes_{F} F_{v} $. 
Then $ h_{F_{v}} $ is isotropic since $ u^{0}(D\otimes_FF_{v}/F_{v})=4 $ \cite[Lem. 4.1(3)]{Wu1}. 

Subcase 2b: Suppose $\lambda$ is not a square in $F_v$ and $ D\otimes_FF_{v} $ is split over the field $ L\otimes_{F} F_{v} $. Then $ h_{F_{v}} $ is Morita equivalent to a hermitian form $ h_{v} $ of rank $ n $ over $ L\otimes_{F} F_{v} $.  
Then $ h_{F_{v}} $ is isotropic iff $ h_{v} $ is isotropic \cite[Ch.I, 9.3.5]{Knus}. 
Also, $ h_{v} $ is isotropic over $  L\otimes_{F} F_{v}  $ iff its trace form $ q_{v} $ is isotropic over $ F_{v} $ by \cite[Ch.10, 1.1(ii)]{Sch}. 
Since the residue field $ \overline{F_{v}} $ is either a $ p $-adic field or a global function field, we have $ u(\overline{F_{v}})=4 $. 
By another theorem of Springer \cite[Ch.6, Cor.2.5]{Sch},  $ u(F_{v})=8 $. 
Since $ q_{v} $ has rank $ 2n>8$, it is isotropic and hence $h_{v}, h_{F_{v}} $ are isotropic. 

Finally, we have the local-global principle for hermitian forms over $ F $ given by \cite[Th.~4.5]{Wu}, so $ h  $ is isotropic. Therefore $ u^0(D \otimes_F L/F) \leq 4 $. 
\end{proof}

\section{Division algebras over semi-global fields}\label{sec5}
%
%We have prove a half of \cref{u}(1) in \cref{semiglobal}(1), now we prove the other half. 
%
%
%\begin{thm}\label{u}
Let $p$ be an odd prime number. 
Let $F$ be the function field of a curve over a $p$-adic field. 
Let $D$ is a division algebra over $F$ with an involution $\sigma$. 
In this section, we show that  the bounds in \cref{semiglobal}  for
the u-invariants of hermitian of forms over central simple algebras over $F$ are 
  in fact exact values.  We also compute $u^0$ if $D$ is a quaternion division algebra 
  with an involution of the second kind over $F$.
  
\begin{lem}\label{key2}	
	 Let $k$ be a  complete discrete 
	 valued field with residue field $\overline{k}$.
	 Suppose $\overline{k}$ is a non-archimedean local field or a global function field with  $\Char(\overline{k}) \ne 2$.  Let $D$ be a division algebra over $k$ with an involution of the first kind and $K/k$ a quadratic extension.  	  
	  
\noindent
 (1) If   $D$ is a quaternion division algebra, 
  then $u^+(D)= 6$ and  $u^-(D)= 2$.

\noindent
(2) If $D$ is a biquaternion algebra, then $u^+(D )= 5$ and $u^-(D )= 3$.

\noindent
(3)  If $D \otimes_k K$ is a division algebra, then $u^0(D\otimes_k K/k) = 4$. 
\end{lem}

\begin{proof}
(1) Suppose $D$ is an unramified quaternion algebra.  Then $\overline{D}$ is a quaternion algebra.
Since $\overline{k}$ is either a local field or a global function field, by \cref{local} and \cref{global}, 
we have $u^+(\overline{D} )=3$, $u^-(\overline{D})=1$ and $u^-(\overline{D})=2$. 
Thus,  by  \cref{Lamour3}(1),  $u^+(D )= 2*3=6$ and $u^-(D )= 2*1=2$.
 
Suppose $D$ is a ramified quaternion algebra. Then $\overline{D}$ is a quadratic extension of 
$\overline{k}$ and  by \cref{local} and \cref{Lamour3}(2), $u^+(D )= 2+4=6$ and $u^-(D )=2+0=2$.

(2) Suppose $D$ is a biquaternion algebra. 
Since $k$ is a complete discrete valued field
with $\overline{k}$ is a global field or local field, $D$ is ramified by a theorem of Albert \cite[III, 4.8]{Lam} and a theorem of Springer \cite[VI, 1.9]{Lam}. 
Thus $\overline{D}$ is a quaternion algebra and hence by \cref{local} and \cref{Lamour3}(2), $u^+(D)= 2+3=5$ and $u^-(D)=2+1=3$. 

(3) Suppose $D \otimes_k K$ is a division algebra.  
Then, by \cref{Lamour4}, we have 
either $u^0(D \otimes_k K/k) =  2u^0(\overline{D \otimes_k K}/\overline{k})$ or  $u^0(D \otimes_k K/k)=  u^0(\overline{D} \otimes_{\overline{k}} \overline{K}/\overline{k}_1)  
+ u^0(\overline{D} \otimes_{\overline{k}} \overline{K}/\overline{k}_3)$ or 
$u^0(D \otimes_k K/k) = u^+(\overline{D_0})+u^-(\overline{D_0})$
for some central division algebra $D_0$ unramified over $k$ with $\deg(D) = \deg(D_0) $.
%By \cref{local} and \cref{global}, we have $u^+(\overline{D}) = 3$, $u^-(\overline{D}) = 1$ and $u^0(\overline{D}) = 2$.

In the case of \cref{Lamour4}(1), $u^0(D \otimes_k K/k) =  2u^0(\overline{D \otimes_k K}/\overline{k})= 2*2=4$; 

In the case of \cref{Lamour4}(2), $u^0(D \otimes_k K/k)=  u^0(\overline{D} \otimes_{\overline{k}} \overline{K}/\overline{k}_1)  
+ u^0(\overline{D} \otimes_{\overline{k}} \overline{K}/\overline{k}_3)$. 
%If $\overline{D\otimes_k K}$ is non-commutative, then $u^0(D\otimes k) = 2+2=4$. 
%If $\overline{D\otimes_k K}$ is commutative, then it is a quadratic extension of $\overline{k}$.
Since $\overline{k}$ is a $p$-adic field or a global field, so are $\overline{k_1}$ and $\overline{k_3}$. 
We have $u(\overline{k_1})=u(\overline{k_3})=4$. 
Since $\overline{D} \otimes_{\overline{k}} \overline{K}$ is a quadratic extension of $\overline{k_1}$, we have 
$u^0(\overline{D} \otimes_{\overline{k}} \overline{K}/\overline{k_1}) = \frac{1}{2}u(\overline{k_1}) = 2$. 
Similarly, $u^0(\overline{D} \otimes_{\overline{k}} \overline{K}/\overline{k_3}) = \frac{1}{2}u(\overline{k_3}) = 2$. 
Thus, we also have $u^0(D\otimes_k K) = 2+2=4$. 

In the case of \cref{Lamour4}(3), $u^0(D \otimes_k K/k) =  u^+(\overline{D_0})+u^-(\overline{D_0})= 3+1=4$.
\end{proof}

%\begin{thm}
%Let $F$ be the function field of a $p$-adic curve with $p \ne 2$ and  $D$   a division 
%algebra over $F$ with an involution of the first kind. Let $L/F$ be a quadratic extension. 
%
%	\noindent
%	(1) If $D$ is quaternion, then $u^+(D)=6$, $u^-(D)=2$.
%
%	
%	\noindent
%	 (2) If $D$ is quaternion and $D \otimes_F L$  is division, then $u^0(D)=4$.
%
%	\noindent
%	  (3) If $D$ is biquaternion, then $u^+(D)=5$ and $u^-(D)=3$.
%\end{thm}
{Now we prove our main result \cref{u}.}
	 
\begin{proof}
 
Since  $D$ is a division algebra. 
By \cite[2.6]{RS}, there exists a divisorial discrete valuation 
$v$ of $F$ such that $D \otimes F_v$ is division. Since $v$ is a divisorial discrete valuation, 
the residue field at $v$ is either a $p$-adic field or a global function field.

(1) and (3) follow from 
  \cref{semiglobal}, \cref{key2}(1)(2) and \cref{dvr}.  
 
(2) By \cite[2.6]{RS}, there exists a divisorial discrete valuation $v$ of
 $F$ such that $D \otimes L \otimes F_v$ is division. 
 Thus, the result follows from 
 \cref{semiglobal2}, \cref{key2}(3) and \cref{dvr}. 
\end{proof}

\section{Tensor product of quaternions over arbitrary fields}\label{app2}
In this section, we revisit and prove \cref{thm2}.  We begin with the following

\begin{lem}\label{5.1}
For $n \ge 1$, let 
$a_n=\frac{4}{5}+\frac{1}{5}(\frac{9}{4})^{n}$, $b_n=-\frac{1}{5}+\frac{1}{5}(\frac{9}{4})^{n}$ and $c_n=\frac{1}{5}+\frac{3}{10}(\frac{9}{4})^{n}$. Then  
\[ a_{n+1}=\frac{3}{4}a_n+c_n, ~b_{n+1}=\frac{3}{2}b_n+\frac{1}{2}c_n,~ c_n=\frac{1}{2}a_n+b_n,~\frac{3}{2}a_n\ge c_n\ge \frac{3}{2}b_n \]
for all $n\ge 1$.

\end{lem}

\begin{proof}
It follows from definitions of $a_n$, $b_n$ and $c_n$ above. 
\end{proof}
%\smallskip
%
%\begin{thm}
%Let $A$ be a central simple algebra over a field $k$.  
%Suppose $\Char k\ne 2$ and $\Per(A)=2$. 
%Suppose $A$ is Brauer equivalent to $H_1 \otimes \cdots \otimes H_n$
%for some quaternion algebras  $H_1, \cdots H_n$ over $k$. 
%%Let $K/k$ be an extension of degree 2. 
%\begin{enumerate}
%\item $u^+(A)\le (\frac{4}{5}+\frac{1}{5}(\frac{9}{4})^{n})u(k)$;
%\item $u^-(A)\le (-\frac{1}{5}+\frac{1}{5}(\frac{9}{4})^{n})u(k)$;
%\item $u^0(A\otimes_k K)\le (\frac{1}{5}+\frac{3}{10}(\frac{9}{4})^{n})u(k)$ 
%for all quadratic extension $K/k$. 
%\end{enumerate}
%\end{thm}

{Now we prove \cref{thm2}.}

\begin{proof}
By \cref{Morita}, we may assume that $A=H_1 \otimes \cdots \otimes H_n$. 
Let $\sigma=\tau_1\otimes\cdots\otimes \tau_n$, where $\tau_i$ is the canonical involutions of $H_i$ for $1\le i\le n$. 
For $n \ge 1$, let 
$a_n=\frac{4}{5}+\frac{1}{5}(\frac{9}{4})^{n}$, $b_n=-\frac{1}{5}+\frac{1}{5}(\frac{9}{4})^{n}$ and $c_n=\frac{1}{5}+\frac{3}{10}(\frac{9}{4})^{n}$.

We proceed by induction. 
For $n =1$, by   \cite[3.4]{Mah} and \cite[2.10]{Leep-sys}  we have   $u^+(H_1)\le a_1u(k)$, by
 \cite[ch.~10, 1.7]{Sch}, we have   $u^-(H_1)\le b_1u(k)$ and by  \cite[4.4]{PS}, we have   $u^0(H_1/k)\le c_1u(k)$. 

Suppose $u^+(H_1\otimes_k\cdots \otimes_kH_n)\le a_nu(k)$, $u^-(H_1\otimes_k\cdots \otimes_kH_n)\le b_nu(k)$ and $u^0(H_1\otimes_k\cdots \otimes_kH_n/k)\le c_nu(k)$. 
%We need to show $u^+(H_1\otimes_k\cdots \otimes_kH_n\otimes_kH_{n+1})\le a_{n+1}u(k)$, $u^-(H_1\otimes_k\cdots \otimes_kH_n\otimes_kH_{n+1})\le b_{n+1}u(k)$ and $u^0(H_1\otimes_k\cdots \otimes_kH_n\otimes_kH_{n+1})\le c_{n+1}u(k)$. 

Let $H_1, \cdots , H_{n+1}$ be quaternion algebas over $k$, $\tau_i$ the canonical involution of $H_i$
and $\sigma = \tau_1 \otimes \cdots \otimes \tau_{n+1}$ on $A = H_1 \otimes \cdots \otimes H_{n+1}$.  
Since $H_{n+1}$ is a quaternion algebra and $\tau_{n+1}$ is the canonical involution,  
 there exist $\lambda_{n+1}, \mu_{n+1} \in H_{n+1}^*$ \ST $\tau_{n+1}(\lambda_{n+1})=-\lambda_{n+1}$, 
 $\tau_{n+1}(\mu_{n+1})=-\mu_{n+1}$, $\lambda_{n+1}\mu_{n+1} =-\mu_{n+1}\lambda_{n+1}$ 
 and $k(\lambda_{n+1})/k$ is a quadratic extension. Let $\lambda = 1 \otimes \cdots \otimes 1 \otimes \lambda_{n+1}
 \in A$, $\mu =   1 \otimes \cdots \otimes 1 \otimes \mu_{n+1}
 \in A$ and $\tilde A$ be the centralizer of $k(\lambda)$ in $A$.
 Then $\tilde{A} =  H_1\otimes  \cdots \otimes H_{n} \otimes k(\lambda)$.
 Let 
 $\sigma_1=\sigma |_{\tilde A}$ and $\sigma_2 = \Int(\mu^{-1})\circ \sigma_1$. 
By \cite[3.1, 3.2]{Mah}, we have 
$\sigma_1$ is unitary, $\sigma_2$ and $\sigma$ are of the same type and 
\[\begin{array}{lll}
u(A, \sigma, \varepsilon) & \le & \min \{ u(\tilde{A}, \sigma_1, \varepsilon) + \frac{1}{2}u(\tilde{A}\otimes k(\lambda), 
\sigma_2, -\varepsilon),  \\
& &  \frac{1}{2}u(\tilde{A}\otimes k(\lambda), 
\sigma_1, \varepsilon) + u(\tilde{A}\otimes k(\lambda), 
\sigma_2, -\varepsilon) \} .
\end{array}
\]
 
%\[\begin{array}{ll}
%& u(H_1\otimes_k\cdots\otimes_k H_{n+1}, \tau, \varepsilon)\\
%\le & \min\{
%u(H_1\otimes_k\cdots\otimes_k H_{n}\otimes_k k(\lambda), \tau_1, \varepsilon)+\frac{1}{2}u(H_1\otimes_k\cdots\otimes_k H_{n}\otimes_k k(\lambda), \tau_2, -\varepsilon),\\
%&\frac{1}{2}u(H_1\otimes_k\cdots\otimes_k H_{n}\otimes_k k(\lambda), \tau_1, \varepsilon)
%+u(H_1\otimes_k\cdots\otimes_k H_{n}\otimes_k k(\lambda), \tau_2, -\varepsilon)\}.\\
%\end{array}
%\]
Since $\sigma_1$ is unitary and $\tilde{A} = H_1\otimes_k\cdots\otimes_k H_{n}\otimes k(\lambda)$,
by the induction hypothesis, we have 
 $u(\tilde{A},\sigma_1, \varepsilon) \le c_n u(k)$.
By \cite[4.2]{PS}, $u(\tilde{A}, \sigma_2, -\varepsilon) = 
u(H_1\otimes_k\cdots\otimes_k H_{n}\otimes k(\lambda), \sigma_2, -\varepsilon)\le \frac{3}{2}u(H_1\otimes_k\cdots\otimes_k H_{n}, \tau_1\otimes  \cdots \otimes  \tau_n, -\varepsilon)$. 

Since both $\sigma$ and $\tau_1\otimes  \cdots \otimes \tau_n$ are   of the first kind and of different types, we have
\[u^+(H_1\otimes_k\cdots\otimes_k H_{n+1})\le \min\{\frac{1}{2}(\frac{3}{2}a_n)+c_n, \frac{3}{2}a_n+\frac{1}{2}c_n\}u(k)=(\frac{3}{4}a_n+c_n)u(k)=a_{n+1}u(k),\]
\[u^-(H_1\otimes_k\cdots\otimes_k H_{n+1})\le \min\{\frac{1}{2}(\frac{3}{2}b_n)+c_n, \frac{3}{2}b_n+\frac{1}{2}c_n\}u(k)=(\frac{3}{2}b_n+\frac{1}{2}c_n)u(k)=b_{n+1}u(k).\]
Finally by \cite[4.3]{PS}, 
$
u^0(H_1\otimes_k\cdots\otimes_k H_{n+1}\otimes_kK/k)\le\min\{\frac{1}{2}a_{n+1}+b_{n+1}, a_{n+1}+\frac{1}{2}b_{n+1}\}u(k)=(\frac{1}{2}a_{n+1}+b_{n+1})u(k)=c_{n+1}u(k) 
$. 
Here \cref{5.1} was used in all three calculations. 
\end{proof}

\begin{rem*}
When $n=2$, $a_2=\frac{29}{16}$ is the same as that of \cite[4.5]{PS}, $b_2=\frac{13}{16}$ is smaller than the bound $\frac{17}{16}$ of \cite[4.6, 4.7]{PS}. 
When $k$ is a semi-global field, $u^-(D)\le\lfloor\frac{13}{2}\rfloor= 6$ is smaller than the bound $8$ of \cite[4.8]{PS}. 

When $n\ge 3$, $a_n$ is smaller than the bound  $\frac{3^{2n-6}}{4^n}\cdot 213$ of \cite[4.10, 4.11]{PS}.
%, since 
%\[
%\begin{array}{lll}
%a_n-\dfrac{3^{2n-6}}{4^n}\cdot 213&=&\dfrac{4}{5}+\dfrac{1}{5}\left(\dfrac{9}{4}\right)^{n}-\dfrac{71}{3^5}
%\left(\dfrac{9}{4}\right)^{n}\\
%&=&\dfrac{4}{5}-\dfrac{112}{3^5\cdot 5}\left(\dfrac{9}{4}\right)^{n}\\
%&\le&\dfrac{4}{5}-\dfrac{112}{3^5\cdot 5}\left(\dfrac{9}{4}\right)^{3}=-\dfrac{1}{4}<0.\\
%\end{array}
%\]
\end{rem*}

%Assume hypotheses of \cref{Lamour}. 
%Although there is no relation between the kind and the type of $\sigma$ and the kind and the type of $\overline{\sigma}$ (see \cite[sec.~1]{CDTWY}), 
%\cref{lcor1} can be refined to the following: 
%
%\begin{cor}\label{lcor2}
%Under \cref{Lamour}, suppose $D$ is a tensor product of quaternion algebras and $u^+(\overline D)>u^-(\overline D)$, then 
%\[u^+(D)=2u^+(\overline{D}),~u^-(D)=2u^-(\overline{D}),~u^0(D)=2u^0(\overline{D}).\]
%\end{cor}
%
%\begin{proof}
%TBA
%\end{proof}

\bibliographystyle{amsalpha}
\bibliography{wu}

\end{document}